\documentclass{amsart}
\usepackage[utf8]{inputenc}

\usepackage[english]{babel}
\usepackage[T1]{fontenc}
\usepackage{amsmath} % � importante che sia prima di amsthm
\usepackage{amsfonts}
\usepackage{amssymb}
\usepackage{amsthm}
\usepackage{mathrsfs}

\usepackage{xypic}
\input xy
\xyoption{all}

\newtheorem{theorem}{Theorem}[section]
\newtheorem{thm}[theorem]{Theorem}
\newtheorem{lemma}[theorem]{Lemma}
\newtheorem{proposition}[theorem]{Proposition}
\newtheorem{corollary}[theorem]{Corollary}

\newtheorem*{thm1}{Theorem \ref{thm:simpson}}%[\ref{thm:simpson}]
\newtheorem*{thm2}{Theorem \ref{thm:pull-back_atiyah}}%[\ref{thm:pullback_atiyah}]

\theoremstyle{definition}
\newtheorem{example}{Example}
\newtheorem{definition}[theorem]{Definition}

\theoremstyle{remark}
\newtheorem{remark}[theorem]{Remark}

\newcommand{\bC}{{\mathbb{C}}}
\newcommand{\C}{{\mathbb{C}}}

\newcommand{\bH}{{\mathbb{H}}}

\newcommand{\Z}{{\mathbb{Z}}}

\newcommand{\id}{{\textbf{1}}}

\newcommand{\corO}{\mathscr{O}}

\newcommand{\gotg}{\mathfrak{g}}

\newcommand{\belA}{\mathcal{A}}

\newcommand{\belE}{\mathcal{E}}

\newcommand{\belG}{\mathcal{G}}

\newcommand{\caL}{\mathcal{L}}

\newcommand{\res}{\text{res}}

\newcommand{\End}{\text{End}}
\newcommand{\Der}{\text{Der}}
\newcommand{\cEnd}{\mathscr{E}nd}

\newcommand{\atiyah}{\text{At}}

\newcommand{\Sym}{\text{Sym}}

\newcommand{\GL}{\text{GL}}

\title{Representations of Atiyah algebroids and logarithmic connections}
\author{Pietro Tortella}
\address{Scuola Internazionale di Studi Avanzati Superiori (SISSA)\\
  via Bonomea 265\\
  34136 Trieste, Italy}
\email{tortella@sissa.it}

\begin{document}

\maketitle

%\section{Introduction}

\noindent
\begin{quote}
\small  {\sc Abstract.}   
In this paper, we investigate representations of $\atiyah(N)$, the Atiyah algebroids of a holomorphic line bundles $N$ over a complex manifold $Y$. In particular, we relate $\atiyah(N)$-modules with logarithmic connections through two functors. On the one hand, we use these functors to the define invariants (monodromy) for representations of Atiyah algebroids. On the other hand, this opens the way to use the theory of Lie algebroids to study problems about logarithmic connections; we will give an example of this by showing that the existence of Deligne's extensions of flat connections and the Riemann-Hilbert correspondence for regular flat meromorphic connections may be obtained as pull-back of similar results for $\atiyah(N)$-modules, and, at this level, these results are a direct consequence of the second theorem of Lie.

\end{quote}

\section{Introduction}

Let $N$ be a line bundle over a complex manifold $Y$. As it was originally introduced in \cite{atiyah}, the Atiyah algebroid of $N$ is an object to describe the connections on $N$: it fits in the exact sequence
$$
0 \to \corO_Y \to \atiyah(N) \to T_Y \to 0
$$
and any splitting of this sequence correspond to a connection on $N$. In the theory of Lie algebroids, Atiyah algebroids play a central role, since they describe transitive Lie algebroids, and allow to formulate the definition of $\caL$-connections and related concepts (cf. \cite{mackenzie}).

The other main topic of this paper are logarithmic connections. Let $X$ be a complex manifold and $D$ a smooth divisor in $X$. Logarithmic connections on $(X,D)$ are connections on $X$ holomorphic on $X\setminus D$ and with poles along $D$, whose behavior is very regular. They can be used to describe the topology of $X\setminus D$, since they are related to  representations of $\pi_1(X\setminus D)$. Moreover, they can be used to describe the Riemann-Hilbert correspondence in a geometric way (cf. \cite{deligne}, \cite{sabbah}).

In this paper, we construct two functors that relate representations of $\atiyah(N)$ with logarithmic connections. Thanks to these functors, one is able to use the theory of Lie algebroids to investigate properties of logarithmic connections, and vice-versa.

The first functor is based on the following remark by Carlos Simpson:
\begin{thm1}
Let $X$ be a smooth complex manifold, and $D$ a smooth divisor in $X$. Consider $T_X(-\log D)$, the sheaf of logarithmic tangent vectors.

Then the restriction $T_X(-\log D)_{|D}$ has a natural structure of Lie algebroid over $D$, and it is isomorphic as a Lie algebroid to $\atiyah(N_{D/X})$, the Atiyah algebroid of the normal bundle to $D$ in $X$.
\end{thm1}
From this, one obtains that the restriction to the divisor $D$ induces a functor
$$
\left\{ \begin{array}{c}
 \text{ logarithmic connection on } X  \\
  \text{with poles along } D
\end{array}\right\}
\longrightarrow
\left\{ \atiyah(N_{D/X})\text{-connection} \right\}\ .
$$
A first consequence is that one obtains a new interpretation of the residue of a logarithmic connection, that can be seen as the action of $\corO_D \hookrightarrow \atiyah(N_{D/X})$ on the restriction $(E,\nabla)_{|D}$ that carries the structure of $\atiyah(N_{D/X})$-module.

The other functor that we construct is based on the following: 
\begin{thm2}
Let $Y$ be a complex manifold and $p:N \to Y$ a line bundle on $Y$. 

Then the pull back $p^* \atiyah(N)$ has a natural structure of Lie algebroid over $N$, and it is isomorphic as a Lie algebroid to $T_N(-\log Y)$, where we see $Y$ as a divisor in $N$ via the zero section embedding.
\end{thm2}

Then one obtains that the pull-back $p^*$ induces a functor at the level of modules
$$
\left\{\atiyah(N)\text{-modules} \right\} \longrightarrow 
\left\{ \begin{array}{c}
 \text{ flat logarithmic connection on } N  \\
  \text{with poles along } Y
\end{array}\right\}\ .
$$
When $X=N$ and $D=Y$, the $p^*$ functor is a left adjoint to the restriction functor. %Moreover, for general $X$, if one wants to study the behavior of objects around the divisor $D$, one can replace $X$ by the total space of $N_{D/X}$, the normal bundle to the divisor.

Using $p^*$, one can define the monodromy of a $\atiyah(N)$-module as the monodromy of the associated logarithmic connection. We then study the Riemann-Hilbert correspondence for $\atiyah(N)$-modules, i.e. we study the existence of $\atiyah(N)$-modules having a given monodromy, and classify them. We will find that at the level of $\atiyah(N)$-modules, the Riemann-Hilbert correspondence is a direct consequence of the second theorem of Lie, and that the Riemann-Hilbert correspondence for regular meromorphic connections may be obtained as a pull-back through $p^*$ of the results that we obtain for $\atiyah(N)$-modules.

%We then use this functor to show that, in the setting of $p:N\to Y$ a line bundle as before, the construction of Deligne's canonical extensions and the Riemann-Hilbert correspondence may be obtained as pull-back of similar results for $\atiyah(N)$-modules, and at the level of $\atiyah(N)$-modules these results follow directly from the second theorem of Lie for Lie algebroids/groupoids. 

Recently, in \cite{gualtieri} have been developed techniques to study connections with singularities through Lie algebroids and groupoids related to the problem, that fit the general ideas of this paper. Moreover , in \cite{gualtieri2} and \cite{guillemin} are presented applications of the logarithmic tangent Lie algebroid  to  symplectic and Poisson geometry.

%The study of problems of connections with singularities with techniques of Lie algebroids and Lie groupoids have been considered in \cite{gualtieri}. Other recent studies of geometric problems with the use of the logarithmic tangent Lie algebroid are \cite{gualtieri2}, \cite{guillemin}.

%in a fashion similar to the well known correspondence between representations of the fundamental group of a manifold and flat bundles over it (cf. \cite{gualtieri}). 

%As it was remarked in \cite{gualtieri}, this correspondence may be seen as a consequence of the second theorem of Lie for Lie algebroids and Lie groupoids: let $M$ be a manifold; then a flat bundle on $M$ is the same thing as a representation of the tangent Lie algebroid $T_M$; by the second theorem of Lie, representations of $T_M$ are in correspondence with representations of the fundamental groupoid $\Pi_1(M)$;  since $\Pi_1(M)$ is a transitive Lie groupoid, its representations are in correspondence with the representations of any of its vertex groups, i.e. representations of $\pi_1(M)$.

\vspace{0.4cm}
The paper is organized as follows: in the next section we recall general definition on Lie algebroids and Atiyah algebroids, we describe the $\textsc{s}$-connected $\textsc{s}$-simply connected Lie groupoid integrating the Atiyah algebroids, and gather basic properties of representation of Atiyah algebroids. These are mainly standard results in Lie algebroid theory, and we gather the ingredients and proofs that we need in order to have a self-contained exposition. In Section 3 we recall the construction of the logarithmic complex and logarithmic connections, prove Theorem \ref{thm:simpson} and explore some direct consequence. In Section 4 we prove Theorem \ref{thm:pull-back_atiyah}, define monodromy for representations of $\atiyah(N)$, study the Riemann-Hilbert correspondence for these objects and show how classical results descend from these.

\vspace{0.4cm}

\noindent {\bf Notation.} 
In this paper, we work in the holomorphic category. Most of the results of Section 2 hold also in the real case, and actually this is how they are usually presented; since in the later sections we will work in the holomorphic category, we will state everything in this setting from the beginning. We will sometime omit the adjectives "complex" or "holomorphic" when they are redundant, but these should be always understood unless otherwise stated.

We use the same letter to denote a holomorphic vector bundle and its sheaf of holomorphic sections. For $p:E \to X$ a holomorphic vector bundle over a complex manifold $X$, when we write $s \in E$  we will always mean that $s$ is a local section of $E$ defined on some open set $U$ in $X$. When more sections are involved, we will assume that they are defined over the same open set. For $x\in X$, we denote by $E(x)$ the fiber of $E$ over $x$. To denote a point $v$ of the fiber bundle $E$, we will use the notation $v \in E(x)$, with $x = p(x)$ (or $v \in E(p(v))$).

\vspace{0.4cm}
\noindent {\bf Acknowledgment.} This paper was mainly motivated by discussions with Carlos Simpson, to whom Theorem \ref{thm:simpson} and the consequences are owed. I am very grateful to him for sharing these and other insightful remarks with me. My many thanks also to Michel Granger, Vladimir Rubtsov, Igor Mencattini and Ugo Bruzzo for important remarks and discussions, and to David Martinez Torres and Marco Zambon for pointing me out related references. The main part of this work was carried out during my visit to the LAREMA  UMR 6093 du CNRS D\'epartement de Math\'ematiques, Universit\'e d'Angers
as a University post-doc 1 year fellow, and subsequently during a post-doc at ICMC, Universidade de S\~ao Paulo, with the FAPESP grant 2011/17593-9. I want to thank both the institutions for the support, and my thanks also to SISSA that is supporting me at the moment.

\section{Atiyah algebroids}

\subsection{Generalities on holomorphic Lie algebroids}

Let $Y$ be a complex manifold. A {\em holomorphic Lie algebroid} on $Y$ is a coherent sheaf of $\corO_Y$-modules $\caL$ with a $\corO_Y$-modules morphism $\rho: \caL \rightarrow T_Y$ called the anchor, and a $\bC_Y$-Lie bracket on its sheaf of sections, satisfying the condition $[u,fv] = f[u,v] + \rho(u)(f) v$ for any $f\in\corO_Y$ and $u,v\in \caL$. For the purposes of this paper, we shall assume that $\caL$ is $\corO_Y$-locally free, i.e. a holomorphic vector bundle over $Y$.

For $\caL$ a holomorphic Lie algebroid, we set $\Omega^p_\caL = \bigwedge^p_{\corO_Y} \caL^*$ and call it the sheaf of $\caL$-forms of degree $p$. The Cartan formula that defines the exterior differential of a manifold can be easily adapted to Lie algebroids, and leads to a differential $d_\caL: \Omega^p_\caL \rightarrow \Omega^{p+1}_\caL$ that squares to $0$. Define the cohomology of the holomorphic Lie algebroid by
$$
H^p(\caL;\bC) = \bH^p(Y; \Omega_\caL^\bullet)\ .
$$

A {\em morphism} between two Lie algebroids $\caL$ and $\caL'$ over the same base is a morphism $\phi:\caL\rightarrow \caL'$ of $\corO_Y$-modules compatible both with the anchors and the brackets. If $\caL$ is a Lie algebroid over $Y$ and $\caL'$ is a Lie algebroid over a different complex manifold $Y'$, a morphism from $\caL$ to $\caL'$  is a pair $(\phi,f)$ with $f:Y\rightarrow Y'$ a morphism between the underlying varieties and $\phi: \caL \rightarrow f^* \caL'$ map of vector bundles over $Y$ such that 
\begin{itemize}
\item[-] 
the diagram
\begin{equation*}
\xymatrix{ \caL \ar[r]^{\phi} \ar[d]_{\rho} & f^* \caL' \ar[d]^{f^* \rho'} \\
T_Y \ar[r]^{T_f} & f^*T_{Y'}
}
\end{equation*}
commutes; 
\item[-] using the identification  $f^* \caL' = f^{-1}\caL' \otimes_{f^{-1}\corO_{Y'}} \corO_Y$ and writing $\phi(v) = \sum \phi(v)_i \otimes a^v_i$ for $v\in \caL,\ \phi(v)_i \in f^{-1}\caL'$ and $a^v_i \in \corO_Y$, the following compatibility between the brackets is satisfied:
$$
\phi([v,w]) = \sum [\phi(v)_i,\phi(w)_j]\otimes a^v_ia^w_j + \rho(v)(a^w_j) \otimes \phi(w)_j - \rho(w)(a^v_i) \otimes \phi(v)_i .
$$
\end{itemize}

\vspace{0.4cm}
Clearly, the holomorphic tangent bundle $T_Y$ to a complex manifold $Y$ is naturally a Lie algebroid, with the identity as anchor and commutator of vector fields as Lie bracket. Similarly, any sub-$\corO_Y$-module $\mathcal F \subseteq T_Y$ with $[\mathcal{F}, \mathcal{F}] \subseteq \mathcal{F}$ has an induced Lie algebroid structure. 

\begin{definition}
Let $ P\stackrel{p}{\to} Y$ be a holomorphic principal $G$-bundle over $Y$, for some complex Lie group $G$ with Lie algebra $\gotg$. Consider the exact sequence obtained from the differential of $p$:
\begin{equation} \label{ses:Tp}
0 \to P \times \gotg \to T_P \stackrel{T_p}{\to} p^*T_Y \to 0\ .
\end{equation}
The group $G$ acts naturally on each element of this sequence, and after taking the quotient one obtains the exact sequence 
\begin{equation} \label{ses:atiyah_principal}
0 \to \operatorname{ad}(P) \to \frac{T_P}{G} \to T_Y \to 0\ .
\end{equation}
of vector bundles over $Y$. The {\em Atiyah algebroid} of $P$ is $\atiyah(P) = T_P/G$, with anchor equal to the quotient map of the sequence \eqref{ses:atiyah_principal}, and bracket induced by the bracket of vector fields in $T_P$.
\end{definition}

Remark that the sequence \eqref{ses:Tp} is an exact sequence of Lie algebroids, where $\operatorname{ad}(P)$ is a Lie algebroid with trivial anchor and $\corO_Y$-linear bracket induced by the Lie algebra structure of $\mathfrak{g}$.

In the case $G = \GL(r,\mathbb{C})$, one can describe the Atiyah algebroid of a $\GL(r,\mathbb C)$-principal bundle in terms of the vector bundle associated to the standard representation of $\GL(r,\mathbb{C})$ on $\mathbb{C}^r$. In fact we have:

\begin{definition} \label{def:atiyah_vb}
Let $E$ be a holomorphic vector bundle over $Y$.

The \emph{Atiyah algebroid} of $E$, denoted by $\atiyah(E)$, is the Lie algebroid on $Y$ defined as follows: 
\begin{itemize}
\item as a vector bundle, $\atiyah(E)$ consists of differential operators on $E$ of order one with scalar symbol, i.e. it is the subsheaf of $\cEnd_{\C_Y}E$ consisting of endomorphisms $D$ such that there exist a vector field $\sigma_D \in T_Y$ (also called the {\em symbol} of $D$) satisfying 
$$
D(fs) = f \cdot D s + \sigma_D(f) \cdot s
$$ 
for any $f\in \corO_Y$ and $s\in E$;
\item the anchor is given by the symbol map, namely $D \mapsto \sigma_D$;
\item the Lie bracket is the commutator of the differential operators.
\end{itemize}
\end{definition}

In this case, we have the following analogue of the sequence \eqref{ses:atiyah_principal}:
\begin{equation} \label{ses:atiyah}
0 \rightarrow \cEnd_{\corO_Y} E \rightarrow \atiyah(E) \rightarrow T_Y \rightarrow 0,
\end{equation}
If $P_E$ is the principal $\GL(r,\bC)$-bundle of frames of the vector bundle $E$, one has a natural isomorphism between the Lie algebroids $\atiyah(P_E)$ and $\atiyah(E)$.

\begin{definition}
Let $\caL$ be a Lie algebroid over $Y$ and $E$ a holomorphic vector bundle on $Y$. A \emph{$\caL$-connection} on $E$ is a $\bC_Y$-modules map $\nabla: E \rightarrow E \otimes \Omega^1_\caL$ satisfying the Leibniz rule $\nabla(fs) = f\nabla s + s \otimes d_\caL f$. Equivalently, a	 $\caL$-connection is a map of bundles $\nabla: \caL \rightarrow \atiyah(E)$ such that the diagram 
\begin{equation*}
\xymatrix{ \caL \ar[rr]^\nabla \ar[rd] && \atiyah(E)\ar[dl] \\ &T_Y &
}
\end{equation*}
obtained from the anchors commutes.

We say that the $\caL$-connection $\nabla$ is {\em flat} (we also say that the pair $(E,\nabla)$ is a \emph{representation } of $\caL$, or that $E$ is a  {\em $\caL$-module}) if moreover $\nabla$ is a morphism of Lie algebroids. 
\end{definition}

\subsection{Lie groupoids and integration}

Definitions and results reported in this subsection are standard in the smooth setting. The holomorphic case was studied in \cite{xu_integration}, where it is obtained that essentially the same results hold in the holomorphic context.

\vspace{0.4cm}
A {\em holomorphic Lie groupoid} $\mathcal G$ over a complex manifold $Y$ is a groupoid (i.e. a small category where each arrow is invertible) such that the space of objects is $Y$, the space of arrows, that we still denote by $\mathcal{G}$, has the structure of complex manifolds, the source and target maps $s,t: \mathcal{G} \to Y$ are holomorphic surjective submersions, the identity $e: Y \to \mathcal{G}$ is a holomorphic embedding, and the multiplication and inverse map are holomorphic maps.

We denote by $\belG_y$, resp. $\belG^y$, the fibers of the source map, resp. target map, i.e. $\belG_y = s^{-1}(y)$, and $\belG^y = t^{-1}(y)$. Introduce the notation $\belG_y^{y'} = \belG_y \cap \belG^{y'}$ for the intersections, and when $y=y'$ the space $\belG_y^y$ has a group structure, that we call {\em vertex group} at $y$. We will also use similar notations with $Z$ any subset of $Y$ in place of $y$.

There is a standard construction that associates to a holomorphic Lie groupoid $\mathcal{G}$ a holomorphic Lie algebroid $\belA(\belG)$. As a vector bundle, $\belA(\belG) = e^* T^s_\belG$, where $T^s_\belG$ is the sub-vector bundle of $T_\belG$ whose sections are the vector fields that are tangent to the fibers of $s$, i.e. sections of the kernel of $T_s: T_\belG \to T_Y$. The anchor on $\belA(\belG)$ is induced by the composition $t \circ e: Y \to Y$, and the bracket is induced by the identification of sections of $\belA(\belG)$ with right-invariant vector fields on $\belG$. 

One says that a holomorphic Lie algebroid $\caL$ is {\em integrable} when there exist a holomorphic Lie groupoid $\belG$ such that $\caL = \belA(\belG)$. Unlike the case of Lie algebras, Lie algebroids are not always integrable (see \cite{mackenzie}, \cite{cf-integrability}). However, one has the following that grants the uniqueness of a "special" Lie groupoid integrating a given integrable Lie algebroid:

\begin{theorem} 
Let $\caL$ be an integrable holomorphic Lie algebroid.

Then there exist a unique holomorphic Lie groupoid $\tilde{\belG}$ such that $\belA(\tilde{\belG}) = \caL$ and the fibers $\tilde{\belG}_y$ are connected and simply connected for any $y$ (in the following we will say that the Lie groupoid with these properties is {\em $\textsc{s}$-connected and $\textsc{s}$-simply connected}).
\end{theorem}

A {\em morphism} between two holomorphic Lie groupoids $\belG$ on $Y$ and $\belG'$ on $Y'$ is a pair $(\Phi,f)$ of morphism of complex manifolds $\Phi: \belG \to \belG'$ and $f: Y \to Y'$ that is compatible with all the structural morphisms of $\belG$ and $\belG'$. We will mainly deal with morphisms of Lie groupoids over the same base $Y$ and such that $f = \id_Y$. When we write $F: \belG \to \belG'$ for a Lie groupoid morphism, it should be understood that it is a morphism of Lie groupoids over the identity.

There is a standard construction that associates to a morphism of Lie groupoids $(\Phi,f):\belG \to \belG'$ a morphism of the associated Lie algebroids $(\belA(\Phi),f): \belA(\belG) \to \belA(\belG')$. Similarly to the case of Lie algebras, the following holds (cf. \cite{LieII}):
\begin{theorem}[Second theorem of Lie] \label{thm:LieII}
Let $\phi: \caL \to \caL'$ be a morphism of holomorphic Lie algeboids over $f: Y \to Y'$. Assume that both $\caL$ and $\caL'$ are integrable, let $\tilde{\belG}$ be the unique $\textsc{s}$-connected and $\textsc{s}$-simply connected Lie algebroid integrating $\caL$, and $\belG'$ any Lie groupoid integrating $\caL'$.

Then there exist a unique morphism of holomorphic Lie groupoids $(\Phi,f): \tilde{\belG} \to \belG'$ such that $\belA(\Phi) = \phi$.
\end{theorem}

\begin{example}[Trivial groupoids]
Let $Y$ be a complex manifold, and $G$ a complex Lie group. Then the trvial Lie groupoid over $Y$ with vertex group $G$ is the groupoid $Y \times G \times Y$, with source (resp. target) map equal to the projection to the first (resp. third) factor, multiplication $(x,g,y) \cdot (y,g',z) = (x,g\cdot g',z)$, identity $e(x) = (x,1,x)$ and inverse $(x,g,y)^{-1} = (y,g^{-1},x)$.

\end{example}

\begin{example}[Pair groupoid]
Let $Y$ be a complex manifold. The {\em pair groupoid} is $X \times X$, the trivial groupoid with group $G = \{1\}$. The Lie algebroid associated to $Y \times Y$ is the tangent Lie algebroid $T_Y$. One has $s^{-1}(y) = \{y\} \times Y$, so that in general $Y\times Y$ is not $\textsc{s}$-simply connected.
\end{example}

\begin{example}[Fundamental groupoid]
Let $Y$ be a connected complex manifold. Define $\Pi_1(Y)$ to be the set of homotopy classes of paths $\gamma:[0,1] \to Y$. By setting $s([\gamma]) = \gamma(0)$, $t([\gamma]) = \gamma(1)$, and product equal to the concatenation of paths, one defines a groupoid structure on $\Pi_1(Y)$, called the {\em fundamental groupoid} of $Y$. The identity $e(y)$ is the constant path at $y$, and the inverse is given by $\gamma^{-1}(t) = \gamma(1-t)$. 

The Lie algebroid associated to $\Pi_1(Y)$ is again the tangent Lie algebroid $T_Y$. The fiber $s^{-1}(y)$ coincide with the space of homotopy classes of loops starting at $y$, that is isomorphic to $\tilde{Y}$, the universal cover of $Y$. So each fiber $s^{-1}(y)$ is simply connected, and $\Pi_1(Y)$ is the $\textsc{s}$-connected and $\textsc{s}$-simply connected Lie groupoid integrating $T_Y$.

Remark that the pair $(s,t): \Pi_1(Y) \to Y\times Y$ (also called the {\em anchor} of the groupoid) defines a Lie groupoids morphism over the identity from the fundamental groupoid to the pair groupoid. This morphism is surjective, and we have the exact sequence of Lie groupoids
$$
1 \to K \to \Pi_1(Y) \to Y\times Y \to 1\ ,
$$
where $K$, is as a totally intransitive Lie groupoid (i.e. a groupoid where $s=t$) with $K_y^y = \pi_1(Y,y)$.
\end{example}

\begin{example}[Gauge groupoid]
Let $p: P \to Y$ be a holomorphic principal $G$-bundle, with $G$ a complex Lie group. 
The {\em gauge groupoid} of $P$ is the holomorphic Lie groupoid $\mathfrak{G}(P)$ over $Y$ defined as follows:
\begin{itemize}
\item the space $\mathfrak{G}(P)$ is $\frac{P\times P}{G}$, with $G$ acting diagonally on the product;
\item the source map is $s([a,b]) = p(a)$, while the target $t([a,b])=p(b)$;
\item the product is $[a,b] \cdot [b',c] = [a,cg]$, where $g\in G$ is such that $b' = bg^{-1}$;
\item the identity is $e(y) = [a,a]$ for any $a\in P(y)$;
\item the inverse is $[a,b]^{-1} = [b,a]$. 
\end{itemize}

The Lie algebroid associated to $\mathfrak{G}(P)$ is $\atiyah(P)$, so Atiyah algebroids of principal and vector bundles are integrable Lie algebroids. The fibers $\mathfrak{G}(P)_y$ are isomorphic to $P$, so in general $\mathfrak{G}(P)$ is not $\textsc{s}$-simply connected.

Remark that the anchor $[a,b] \mapsto (p(a),p(b))$ yields the exact sequence of Lie groupoids
$$
1 \to P\times_G G \to \mathfrak{G}(P) \to Y \times Y \to 1\ ,
$$
where the leftmost term is the fiber bundle associated to the adjoint representation of $G$ on itself, seen as a totally intransitive Lie groupoid with $s=t$ equal to the bundle projection map.

\end{example}

\begin{example}[Gauge-Path groupoid] \label{ex:GP-groupoid}
Let $p:P \to Y$ be as before; define the following Lie groupoid $\mathfrak{C}(P)$ over $Y$: 
\begin{itemize}
\item arrows in $\mathfrak{C}(P)$ are elements of the quotient $\Pi_1(P)/G$, where the action of $G$ on paths is given by $(\gamma \cdot g)(t) = \gamma(t)\cdot g$;
\item source $s([\gamma]) = p(\gamma(0))$, and target $t([\gamma]) = p(\gamma(1))$;
\item the product $\gamma\cdot \gamma'$ is given by the concatenation of $\gamma$ with $\gamma'\cdot g$, where $g\in G$ is such that $\gamma(1) = \gamma'(0) \cdot g^{-1}$;
\item the identity over $y$ is the class of the constant path at $a$ for any $a$ with $p(a)=y$; the inverse is defined as the inverse in the path groupoid.
\end{itemize}

Since the Lie algebroid associated to $\Pi_1(P)$ is $T_P$, the Lie algebroid associated to $\mathfrak{C}(P)$ is $T_P/G = \atiyah(P)$. The fibers $\mathfrak{C}(P)_y$ are isomorphic to $\Pi_1(P)_a$ for any $a \in P(y)$. So $\mathfrak{C}(P)_y$ is isomorphic to $\tilde{P}$, the universal cover of $P$, and as such it is simply connected. Thus $\mathfrak{C}(P)$ is the unique $\textsc{s}$-connected $\textsc{s}$-simply connected Lie algebroid integrating $\atiyah(P)$.

The groupoid $\mathfrak{C}(P)$ has two natural surjective groupoid morphisms over the identity: one is $p:\mathfrak{C}(P) \to \Pi_1(Y)$ induced by the bundle map $p$, and the other $\epsilon : \mathfrak{C}(P) \to \mathfrak{G}(P)$ given by taking the endpoints of a path. The kernel of $p$ is isomorphic to the fiber bundle $P \times_G \tilde{G}$, where one identifies $\tilde{G}$ with the space of homotopy classes of paths in $G$ starting at the identity, and the $G$-action on $\tilde{G}$ is $(g\cdot \gamma)(t) = g \cdot \gamma(t) \cdot g^{-1}$.

These give the commutative diagram of Lie groupoids
\begin{equation} \label{diagram0}
\xymatrix{ 
P \times_G \pi_1(G,1) \ar[d] \ar[r] & P \times_G \tilde{G} \ar[d] \ar[r] & P \times_G G \ar[d] \\
W \ar[d] \ar[r] & \mathfrak{C}(P) \ar[d]^p \ar[r]^\epsilon & \mathfrak{G}(P) \ar[d] \\
K \ar[r] & \Pi_1(Y) \ar[r] & Y\times Y
}
\end{equation}
Remark that both $W$ and $K$ are totally intransitive Lie groupoids over $Y$, with $K_y^y = \pi_1(Y,y)$, and $W^y_y = \pi_1(P, a)$ for any $a \in P(y)$.

\end{example}

\begin{example}[Frame groupoid]
Let $E$ be a vector bundle on $Y$. Define $\operatorname{Iso}(E)$, the {\em frame groupoid} of $E$, to be the Lie groupoid whose elements are triples $(y_1,y_2,\phi)$, with $y_1,y_2 \in Y$ and $\phi: E(y_1) \to E(y_2)$ an isomorphism. The Lie groupoid structure is given by $s(y_1,y_2,\phi)= y_1$ and $t(y_1,y_2,\phi) = y_2$, the product $(y_1,y_2,\phi)(y_2,y_3,\phi') = (y_1,y_3,\phi' \circ \phi)$, the identity $e(y) = (y,y, \id_{E(y)})$ and the inverse $(y_1,y_2,\phi)^{-1}=(y_2,y_1,\phi^{-1})$.

Remark that if $P_E$ is the principal $\GL(r,\mathbb C)$-bundle of frames of $E$, one has a natural isomorphism of Lie groupoids between $\operatorname{Iso}(E) $ and the gauge groupoid $\frac{P_E \times P_E}{\GL(r,\mathbb{C})}$.
\end{example}

In the previous section, we gave a definition of representation of a Lie algebroid $\caL$ on a vector bundle $E$ as a morphism of Lie algebroids $\caL \to \atiyah(E)$. Similarly, one has the following:

\begin{definition}
Let $G$ be a Lie groupoid and $E$ a vector bundle on $Y$. A  {\em $G$-module structure} on $E$, or a {\em representation} of $G$ on $E$ is a Lie groupoid morphism $\Phi : G \to \operatorname{Iso}(E)$ over the identity. 
\end{definition}

Then, one obtains the following corollary of the second theorem of Lie:
\begin{corollary} \label{cor:representations}
Let $\caL$ be an integrable holomorphic Lie algebroid. 

There is an equivalence between the category of $\caL$-modules and the category of $\tilde{\belG}$-modules, where $\tilde{\belG}$ is the unique $\textsc{s}$-connected and $\textsc{s}$-simply connected Lie groupoid integrating $\caL$.
\end{corollary}

A Lie groupoid is said {\em transitive} when the anchor $(s,t)$ is surjective. It turns out (cf. \cite{mackenzie}, Chapter 1.5) that a transitive Lie groupoid is locally trivial, i.e. that one can always find local sections of $t_{|\belG_y} : \belG_y \to Y$. Given $\alpha: U \to \belG_y$ a section of $t_{|\belG_y}$ defined on some open set $U\subseteq Y$, one can trivialize $\belG$ over $U$: one can write $g = \alpha(s(g))^{-1} \gamma \alpha(t(g))$, where $\gamma = \alpha(s(g)) g \alpha(t(g))^{-1} \in \belG^y_y$, and define an isomorphism $\belG_U^U \stackrel{\tilde{\alpha}}{\to} U \times \belG^y_y \times U$ by $\tilde{\alpha}(g) = (s(g), \gamma , t(g))$. One can use the local triviality of a transitive Lie groupoid $\belG$ to build a representation of $\belG$ from a representation of its vertex groups. Namely one has:

\begin{theorem} \label{thm:morita}
Let $\belG$ be a transitive Lie groupoid over $Y$, and fix $y\in Y$.

Then there is an equivalence between 
\begin{itemize}
\item[-] equivalence classes of representations $\belG \to \operatorname{Iso}(E)$ into a rank $r$ vector bundle;
\item[-] conjugacy classes of representations of $\belG^y_y \to \GL(r,\mathbb C)$.
\end{itemize}
\end{theorem}
\begin{proof}
%Remark that $\operatorname{Iso}(E)$ is a transitive Lie groupoid, and the vertex groups $\operatorname{Iso}(E)_y^y$ are isomorphic to $\GL(r,\mathbb C)$. For $G$ and $H$ complex Lie groups, the Lie algebroid morphisms over the identity between the trivial Lie algebroids $Z \times G \times Z$ and $Z \times H \times Z$ are of the form $\Phi(z_1,g,z_2) = (z_1, A(z_1)^{-1} \theta(g) A(z_2) , z_2)$ for $A: Z \to H$ a holomorphic map and $\theta:G \to H$ a complex Lie groups morphism. The choice of $A$ and $\theta$ is not unique, and it is determined by the choice of $z_0 \in Z$, by setting $\Phi(z_0 , 1 , z) = (z_0 , A(z) , z)$ and $\Phi(z_0,g,z_0) = (z_0, \theta(g) , z_0)$. Choosing another $\tilde{z}_0$ links the associated group morphisms via the formula $\tilde{\theta} = \operatorname{ad}_{A(z_0)^{-1}} \theta$.

The restriction to $y$ yields the functor
$$
\belG\text{-modules} \longrightarrow \belG^y_y\text{-modules}\ 
$$
that associates to a $\belG$-module $E$ the representation $\rho_{E(y)}$. One need to find an inverse to this, i.e. associate to any representation $\rho:\belG^y_y \to \GL(r,\mathbb C)$ a $\belG$-module $S_\rho$ such that  $S_{\rho_{E(y)}} = E$ and  $S_\rho$ restricted to $y$ yields back $\rho$.

Recall that $\belG_y^y$ naturally acts on the source fiber $\belG_y$, and, when $\belG$ is transitive, each $\belG_y$ is a principal $\belG_y^y$-bundle over $Y$, with projection $\belG_y \to Y$ equal to the target map $t_{|\belG_y}$, and equipped with a natural framing $(t_{|\belG_y})^{-1}(y) = \belG^y_y$. 

Then, consider $S_\rho = \belG_y \times_{\belG^y_y} \mathbb C^{r}$, the vector bundle associated to the principal bundle $\belG_y$ and the representation $\rho$. 

Any $g_1 \in \belG_y$ defines the representation $\operatorname{ad}_{g_1} \rho :\belG^{y_1}_{y_1} \to \GL(r, \mathbb C)$, with $y_1 = t(g_1)$, and $\operatorname{ad}_{g_1} \rho \ (c) = \rho(g_1 c g_1^{-1})$. There is a natural isomorphism $\Theta_{g_1} : S_\rho \to S_{\operatorname{ad}_{g _1}\rho}$ defined by $\Theta_{g_1}([a,v]) = [g_1^{-1}a , v]$. 

%If $g_1'$ is another element in $G_y^{y_1}$, one has $\operatorname{ad}_{g_1'} \rho = \operatorname{ad}_{\rho(\lambda)} \cdot \operatorname{ad}_{g_1}$, where $\lambda = g_1' g_1^{-1} \in G^y_y$, and $\Theta_{g_1'} = \Theta_{g_1} \circ \Theta_\lambda$. 

Now, let $g \in \belG$ be any arrow with $s(g) = y_1$ and $t(g) = y_2$, choose a $g_1 \in \belG_y^{y_1}$ and set $g_2 = g_1g \in \belG_y^{y_2}$. Define the isomorphism $\Phi_g : S_\rho(y_1) \to S_\rho(y_2)$ as the composition
$$
S_{\rho}(y_1) \stackrel{\Theta_{g_1}}{\longrightarrow} S_{\operatorname{ad}_{g_1} \rho} (y_1) \longrightarrow  S_{\operatorname{ad}_{g_2} \rho} (y_2) \stackrel{\Theta_{g_2}^{-1}}{\longrightarrow} S_{\rho}(y_2)
$$
where the central isomorphism is defined by using the framings of $\belG_{y_1}$ and $\belG_{y_2}$, explicitly $[e(y_1) , v]\mapsto [e(y_2) , v]$. One checks that $\Phi_g$ does not depend on the choice of the arrow $g_1$, so that $g \mapsto \Phi_g$ defines a structure of $\belG$-module on $S_\rho$. 

It is clear by its construction that the restriction of $S_\rho$ to $y$ coincides with $\rho$ and that $S_{\rho_{E(y)}} = E$.

\end{proof}

\subsection{Modules of the Atiyah algebroid of a line bundle} \label{sec:atiyah_lb}

Let us gather here some properties of $\atiyah(N)$-modules, for $N$ a line bundle over $Y$. Remark that in this case the sequence \eqref{ses:atiyah} becomes
\begin{equation} \label{ses:atiyah_lb}
0 \to \corO_Y \to \atiyah(N) \to T_Y \to 0\ .
\end{equation}
We shall denote by $R$ the first inclusion, and call its image in $\caL_N$ the {\em isotropy} of the Lie algebroid $\atiyah(N)$. Since $\corO_Y$ is abelian as a Lie algebroid, one has the following (see \cite{BB}, \cite{tesi}):

\begin{proposition} \label{prop:la_extension}
The Lie algebroid extensions of $T_Y$ by $\corO_Y$ are classified by the group $\mathbb {H}^2(Y; \tau^{\geq 1} \Omega_Y^\bullet)$, where $\tau^{\geq \cdot}$ denotes the truncation of the complex.

Under this correspondence, the Atiyah algebroid $\atiyah(N)$ corresponds to the Chern class of $N$. 

In particular, $\atiyah(N) = T_Y \oplus \corO_Y$ if and only if $c_1(N)=0$.
\end{proposition}

Let now $V$ be a vector bundle on $Y$, and $\delta$ a $\atiyah(N)$-connection on $V$; then 
$$
\delta: V \rightarrow V \otimes \Omega_{\atiyah(N)}^1\ 
$$
is a $\mathbb{C}_Y$-modules morphism satisfying the Leibniz rule $\delta(fs) = f\delta s + s \otimes d_{\atiyah(N)} f$. 
%Since $\atiyah(N)$ is a transitive Lie algebroid over $Y$, a necessary condition for existence of $\atiyah(N)$-connections is that $V$ must be $\corO_Y$-locally free.

\begin{remark}
Here and in the following, the $\atiyah(N)$-connections are holomorphic structures on $N$. If one wants to work with smooth structures, one should be careful to choose the appropriate real Lie algebroids associated to $\atiyah(N)$. For example, to reformulate the following propositions using objects in the real category one should use the Lie algebroid $\atiyah(N)_{\mathbb{R}} \bowtie T^{0,1}_X$ as in \cite{xu}.
\end{remark}

Dualizing the sequence \eqref{ses:atiyah_lb} one obtains
\begin{equation} \label{ses:atiyah_dual}
0\rightarrow \Omega_Y^1 \rightarrow \Omega_{\atiyah(N)}^1 \stackrel{R^*}{\rightarrow} \corO_Y \rightarrow 0\ ,
\end{equation}
and by composing $\delta$ with the quotient map $R^*$ one obtains a $\corO_Y$-linear endomorphism $\delta_R \in \End_{\corO_Y} V$. We call $\delta_R$ the \emph{isotropy action (or endomorphism)} of $\delta$ on $V$.

The next proposition gives some restriction on the topology of $V$ for the existence of $\atiyah(N)$-connections on it:

\begin{proposition} \label{prop:divisible}
Let $V$ be a vector bundle on $Y$.

If there exist $\atiyah(N)$-connections on $V$, then any characteristic class of $V$ is divisible by $c_1(N)$, the Chern class of $N$.
\end{proposition}

This fact follows from a generalization of Atiyah's theory on the existence of holomorphic connections to the Lie algebroid setting: it is possible to show the following (see \cite{tesi}):
\begin{proposition}
Let $V$ be a vector bundle on $Y$. Denote by $C(V)\subseteq H^\bullet (Y; \bC)$ the characteristic ring of $V$, i.e. the subring generated by the characteristic classes of $V$.
Let $\caL$ be a holomorphic Lie algebroid, and denote by $\psi$ the natural morphism $H^\bullet(Y;\bC)\rightarrow H^\bullet (\caL;\bC)$ and $C^\caL(V) = \psi(C(V)).$

If $V$ admits a $\caL$-module structure then $C^\caL(V) = 0$.
\end{proposition}

Now to prove Proposition \ref{prop:divisible} one needs to study the cohomology of $\caL_N$. One has the following (see \cite{tutti}):

\begin{proposition}
The complex $\Omega^\bullet_{\atiyah(N)}$ fits in the exact sequence of complexes
$$
0\rightarrow \Omega_Y^\bullet \rightarrow \Omega^\bullet_{\atiyah(N)} \rightarrow \Omega^{\bullet -1}_Y \rightarrow 0
$$
that induces a long exact sequence in cohomology
$$
\cdots \rightarrow H^p(Y;\bC) \rightarrow H^p(\atiyah(N);\bC) \rightarrow H^{p-1}(Y;\bC) \rightarrow H^{p+1}(Y;\bC) \rightarrow\cdots
$$
where the connecting morphism is given by the cup product with $c_1(N)$.

In particular, for the cohomology of $\atiyah(N)$ we have the isomorphisms:
$$
H^k(\atiyah(N) ; \bC) = \operatorname{Ker} \gamma_{|H^{k-1}(Y; \bC)} \oplus \frac{H^k(Y;\bC)}{\operatorname{Im} \gamma_{H^{k-2}(Y; \bC)}}\ ,
$$
where $\gamma : H^\bullet(Y;\bC) \to H^{\bullet+2}(Y; \bC)$ is the cup product with $c_1(N)$.  
\end{proposition}

%Define $F^1 \subseteq \Omega_{L_N}^p$ to be the subsheaf consisting of forms $\eta$ such that $i_R \eta =0$, where $i_\bullet$ for $\bullet \in L_N$ denotes the contraction and $R\in L_N$ is the image of $1\in \corO_Y \hookrightarrow L_N$. Then $F^1= \Omega_Y^p$, and the quotient $\Omega^p_{L_N}/F^1$ is naturally isomorphic to $\Omega^{p-1}_Y$. From this the exact sequence of complexes follows.

%To compute the connecting morphism one should choose local splitting of this sequence and see how they glue together. But splittings of this sequence are the same thing as connections on $N$, so the proposition follows.

\vspace{0.4cm}
Now let us examine further properties of $\atiyah(N)$-modules, i.e. $\atiyah(N)$-connections that are flat.

To begin, one has the following basic examples of $\atiyah(N)$-modules:

\begin{itemize}
\item any $T_Y$-module is naturally a $\atiyah(N)$-module, with $\delta_R = 0$. Conversely any $\atiyah(N)$-module with $\delta_R = 0$ is naturally a $T_Y$-module. In particular, $\corO_Y$ is a $\atiyah(N)$-module;
\item for any $a\in \Z$, $N^{\otimes a}$ is a $\atiyah(N)$-module, with $\delta_R = a \cdot \id_{N^{\otimes a}}$;
\item for $V$ a $\atiyah(N)$-module and $a\in\Z$, the tensor product $V \otimes N^{\otimes a}$ is naturally a $\atiyah(N)$-module, and $\delta_R^{V\otimes N^{\otimes a}} = \delta_R^V \otimes \id_{N^{\otimes a}} + a\cdot\id_{V \otimes N^{\otimes a}}$.
\end{itemize}

The following proposition can be seen as a prequel to the correspondence between $\atiyah(N)$-modules and logarithmic connections that we will establish in the following sections:

\begin{proposition} \label{prop:rigidity}
Let $(V,\delta)$ be a $\atiyah(N)$-module, and $\delta_R\in \End(V)$ the isotropy endomorphism. Then the endomorphisms of the fibers $\delta_R(y) \in \End(V(y))$ belong to the same conjugacy class for any $y\in Y$.

In particular, the eigenvalues of $\delta_R(y)$ do not depend on the point $y$.
\end{proposition}

\begin{proof}

Let $y_0\in Y$, and choose $\zeta$ a local splitting of the sequence \eqref{ses:atiyah_lb} in a neighborhood of $y_0$, such that $\zeta$ is a morphism of Lie algebroids. %This is possible since the choice of a splitting $\zeta$ is equivalent to the choice of a connection on the line bundle $N$, and to require that $\zeta$ is a morphism of Lie algebroid is equivalent to requiring that the associated connection is flat.

Define $\delta_\zeta = \delta \circ \zeta$. We have $\delta = \delta_R + \delta_\zeta$, with $\delta_\zeta$ a flat $T_Y$-connection on $V$. For any other point $y$ in the neighborhood where $\zeta$ is defined, the parallel transport defined by $\delta_\zeta$ induces an isomorphism $T: V(y_0) \rightarrow V(y)$. 

\emph{Claim}: $T$ conjugates the endomorphisms $\delta_R(y_0)$ and $\delta_R(y)$.

\noindent
The flatness of $\delta$ implies that $\delta_R\circ \delta_\zeta = \delta_\zeta \circ \delta_R$.
Let $v\in V(y_0)$, and $s$ be the unique $\delta_\zeta$-horizontal section with $s(y_0)=v$. Then $T(v) = s(y)$ and
$$
\delta_R(y)(T(v)) = \delta_R(y)(s(y)) = (\delta_Rs)(y)
$$
while on the other hand
$$
T(\delta_R(y_0)(v)) = T((\delta_Rs)(y_0)) = (\delta_Rs)(y)\ .
$$
The last equality holds because, since $s$ is $\delta_\zeta$-horizontal, $\delta_R s$ is $\delta_\zeta$-horizontal as well, and it defines the parallel transport of $\delta_R(y_0)(v)$.

\vspace{0.3cm}

For $y$ a point far from $y_0$, join $x$ to $x_0$ with a smooth path, cover the path with a finite number of open sets over which splittings as before exist, and iterate the previous argument.

\end{proof}

\begin{proposition} \label{prop:decomposition}
Let $(V,\delta)$ be a $\atiyah(N)$-module, and let $V = \bigoplus_\lambda V_\lambda$ be the decomposition of $V$ in generalized eigenspaces for $\delta_R$, i.e. $V_\lambda = \ker(\delta_R-\lambda \id)^a$ for some integer $a$ large enough.

Then each $V_\lambda$ is a sub-$\atiyah(N)$-module of $V$ and the decomposition $V=\bigoplus_\lambda V_\lambda$ is a decomposition as a direct sum of $\atiyah(N)$-modules.

In particular, if $V$ is an irreducible $\atiyah(N)$-module, then its isotropy endomorphism $\delta_R$ has only one eigenvalue.
\end{proposition}
\begin{proof}
The flatness of $\delta$ implies that $\delta$ commutes with $\delta_R$. On the other hand, $\delta$ commutes with the multiplication by $\lambda$, since $\lambda$ is a complex number. Then $(\delta_R -\lambda)^a \circ \delta = \delta \circ (\delta_R -\lambda)^a$, so if $v \in V_\lambda$ we have $\delta v \in V_\lambda \otimes \Omega_{\atiyah(N)}$ as well, i.e. $V_\lambda$ is a sub-$\atiyah(N)$-module of $V$, and the proposition follows. 
\end{proof}

\section{Restriction of logarithmic connections}

\subsection{The logarithmic complex}

In this subsection we recall briefly some definition and basic properties on the logarithmic tangent complex and logarithmic connections. Since our later results hold (at the moment) only for divisors that are smooth, we will simplify the exposition and assume from the beginning that the divisor $D$ where the singularities take place is smooth, whereas the results of this subsection are usually formulated for $D$ a divisor with simple normal crossing singularities. For more details, see for example \cite{deligne}, \cite{ev}.

\vspace{0.4cm}
Let $X$ be a complex manifold, and $D\subseteq X$ a smooth divisor. 
\begin{definition}
The sheaf of {\em $p$-forms with logarithmic poles along $D$} is the subsheaf $\Omega^p_X(\log D) \subseteq \Omega^p_X(D)$ which consists of the $p$-forms $\eta \in \Omega^p_X(D)$ such that $d\eta \in \Omega^{p+1}_X(D)$. 
\end{definition}

Clearly, the differential of a form with logarithmic poles has again logarithmic poles, so $(\Omega^\bullet_X(\log D) , d)$ forms a complex of sheaves. 

One has the following:
%for $D$ a divisor with simple normal crossing singularities:
\begin{proposition}
\begin{enumerate}
\item the sheaves $\Omega^p_X(\log D)$ are $\corO_X$-locally free, and one has isomorphisms $\Omega^p_X(\log D) = \bigwedge^p \Omega^1_X(\log D)$;
\item if $x_1,\ldots,x_n$ are local coordinates on $X$ such that $D$ has equation $\{x_1  = 0\}$, a frame of $\Omega_X^1(\log D)$ is given by $\{\frac{d x_1}{x_1},  d x_{2}, \ldots, d x_n\}$.
\end{enumerate}
\end{proposition}

%One can define the sheaves of forms with logarithmic poles along $D$ $\Omega^1_X(\log D)$ of \emph{differentials with logarithmic sigularities along $D$} to be the subsheaf of $\Omega_X(D)$ of forms $\alpha$ such that both $\alpha$ and $d \alpha$ have simple poles along $D$. It is a locally free $\corO_X$-module, and locally, if $t$ is an equation of $D$ and we choose $x_1,\ldots,x_{n-1},t$ coordinates for $X$, it is $\corO_X$-generated by $d x_1, \ldots, d x_{n-1}, dt/t$.

The sheaf $\Omega^1_X(\log D)$ contains the sheaf of holomorphic forms $\Omega^1_X$, and one has the exact sequence
\begin{equation} \label{ses:omega_log}
0 \rightarrow \Omega^1_X \rightarrow \Omega^1_X(\log D) \rightarrow O_D \rightarrow 0\ .
\end{equation}
The quotient map is the Poincar\'e residue, that we denote by $R_D$, and may be described explicitly as follows: for $\alpha \in \Omega^1_X(\log D)$ write $\alpha = \alpha^h + \alpha^0 \frac{dx_1}{x_1}$ with $\alpha^h \in \Omega^1_X$ and $\alpha^0 \in \corO_X$; then
\begin{equation}
R_D \left( \alpha^h + \alpha^0 \frac{dx_1}{x_1} \right) = \alpha^0 \: _{|D}\ .
\end{equation}

\vspace{0.4cm}
A {\em connection with logarithmic poles along $D$} (or simply a {\em logarithmic connection} when no confusion on the divisor $D$ may arise) is a pair $(E, \nabla)$ with $E$ a holomorphic vector bundle on $X$ and 
$$
\nabla: E \to E \otimes \Omega^1_X(\log D)
$$
a map of $\mathbb C_X$-modules satisfying the Leibniz rule $\nabla(fs) = f \cdot \nabla s + s \otimes df$. One introduces the curvature of a logarithmic connection in the usual way, and say that the connection is {\em flat} when its curvature vanishes.

Let $(E,\nabla)$ be a logarithmic connection. One defines the {\em residue} of $\nabla$ along $D$ to be the composition 
$$
E \stackrel{\nabla}{\to} E \otimes \Omega^1_X(\log D) \stackrel{R_D}{\to} E\otimes \corO_D \ .
$$
One checks that the composition is $\corO$-linear, and that it vanishes on the sections in $E(-D)$, so that it defines an endomorphism $\res_D\nabla \in \End_{\corO_D}(E_{|D})$.

\vspace{0.4cm}
From the Lie algebroid point of view, one can rephrase the previous paragraphs as follows: let us denote be $T_X(- \log D)$ the vector bundle dual to $\Omega^1_X(\log D)$. This is a sub-$\corO_X$-module of the tangent bundle $T_X$, and is described as the set of derivations $V \in T_X = \Der(\corO_X)$ that preserve $I_D$, the ideal sheaf of $D$, i.e. such that $V(I_D) \subseteq I_D$. Equivalently, it consists of those vector fields on $X$ that are tangent to $D$, and if $x_1,\ldots,x_n$ are coordinates on $X$ such that $D$ has equation $x_1=0$, then a local frame of $T_X(- \log D)$ is given by $\{ x_1\cdot \partial_{x_1}, \partial_{x_2}, \ldots , \partial_{x_n} \}$.

The commutator of two vector fields in $T_X(- \log D)$ is again in $T_X(-\log D)$, so $T_X(- \log D)$ is a sub-Lie algebroid of $T_X$. The following is straightforward:
\begin{proposition}
Let $X$ be a complex manifold, and $D$ a smooth divisor in $X$.

Then the complex $\Omega^\bullet_{T_X(-\log D)}$ is isomorphic to $\Omega_X^\bullet (\log D)$.
\end{proposition}

So logarithmic connection are the same thing as $T_X(-\log D)$-connection, and flat logarithmic connection are equivalent to $T_X(-\log D)$-module.

\subsection{Restriction of logarithmic objects to the divisor of singularities}

The tangent logarithmic bundle fits in the exact sequence
\begin{equation} \label{ses:tangent_log}
0 \rightarrow T_X(- \log D) \rightarrow T_X \rightarrow N_{D/X} \rightarrow 0\ ,
\end{equation}
where $N_{D/X}$ is the normal bundle to $D$ in $X$. Denote by $L_{D/X}$ the restriction of $T_{X} (-\log D)$ to the divisor $D$. By restricting the sequence \eqref{ses:tangent_log} to the divisor $D$ one obtains the exact sequence of $\corO_D$-modules
\begin{equation} \label{eqn:4exact}
0 \rightarrow \corO_D \rightarrow L_{D/X} \rightarrow T_{X|D} \rightarrow N_{D/X} \rightarrow 0\ .
\end{equation}
The image of the central arrow coincides with $T_D$, and the sequence splits into the two short exact sequences 
\begin{equation} \label{ses:split}
0 \rightarrow \corO_D \rightarrow L_{D/X} \rightarrow T_D \rightarrow 0 \quad ; \quad 0 \rightarrow T_D \rightarrow T_{X|D} \rightarrow N_{D/X} \rightarrow 0\ .
\end{equation}

One has the following:

\begin{thm} \label{thm:simpson}
The $\corO_D$-module $L_{D/X} = T_X(-\log D)_{|D}$ inherits a structure of Lie algebroid over $D$.

As a Lie algebroid, $L_{D/X}$ is isomorphic to $\atiyah(N_{D/X})$, the Atiyah algebroid of the normal bundle $N_{D/X}$.
\end{thm}

\begin{proof}
The anchor of $L_{D/X}$ is given by the quotient map of the first of the exact sequences in \eqref{ses:split}, while the bracket is defined by taking lifts to $T_X(-\log D)$ of the sections of $L_{D/X}$. This is well defined since one has
\begin{equation} \label{eqn:fund_log_tangent}
[T_X(-\log D), T_X(-D)] \subseteq T_X(-D)\ .
\end{equation}

To prove that $L_{D/X}$ is isomorphic to $\atiyah(N_{D/X})$, first notice that these are both Lie algebroid extensions of $T_D$ by $\corO_D$. Thus to prove the isomorphism it suffice to define a map of extensions from $L_{D/X}$ to $\atiyah(N_{D/X})$, and to do this means to define an action of $L_{D/X}$ on the sections of $N_{D/X}$ by means of differential operators living in $\atiyah(N_{D/X})$.

Let $\xi \in L_{D/X}$ and $\gamma \in N_{D/X}$. Choose $w \in T_X(-\log D)$ such that $w_{|D} = \xi$. Remark that $(w + w_0)_{|D} = \xi$ if and only if $w_0\in T_X(-\log D)(-D)$. Denote by $\pi$ the projection $T_{X|D}\rightarrow N_{D/X}$. Let $\bar{\gamma}$ be a section of $T_{X|D}$ with $\pi(\tilde{\gamma}) = \gamma$, and $v \in T_X$ a section such that $v_{|D} = \tilde{\gamma}$. Remark that an element $\pi((v + v_0)_{|D}) = \gamma$ if and only if $v_{0|D} \in T_D$, that is if and only if $v_0 \in T_X(- \log D)$.

Define the action of $\xi$ on $\gamma$ to be  $\nabla_\xi \gamma = \pi([w,v]_{|D})$, where the bracket is the commutator of vector fields of $T_X$. One should check that 
\begin{enumerate}
\item this definition does not depend on the choices, i.e. $\pi([w_0,v]_{|D}) = 0 = \pi([w,v_0]_{|D})$;
\item $\nabla_{f\xi} \ \gamma = f \nabla_\xi \gamma$ for any $f \in \corO_D$;
\item $\nabla_\xi(f \gamma) = f\nabla_\xi \gamma + \xi(f)\cdot \gamma$ for any $f \in \corO_D$.
\end{enumerate}

The point $1.$ is straightforward: $\pi([w_0,v]_{|D}) = 0$ if and only if $[w_0,v] \in T_X(- \log D)$, and this is true since $[T_X(- \log D)(-D), T_X] \subseteq T_X(- \log D)$, while $[w,v_0] \in T_X(- \log D)$ since both $w$ and $v_0$ live in $T_X(- \log D)$.

For the point $2.$, let $\tilde{f}$ be any function in $\corO_X$ such that $\tilde{f}_{|D} = f$; we have 
\begin{align*}
\nabla_{f\xi} \ \gamma & = \pi([\tilde{f}w,v]_{|D}) \\
    &= f\pi([w,v]_{|D}) - \pi((v(f)\cdot w)_{|D})
\end{align*}
and the last summand is zero since $w_{|D}\in T_D$. 

Similarly, for $3.$ we have: 
\begin{align*}
\nabla_{\xi} \ f\gamma & = \pi([w, \tilde{f} v]_{|D}) \\
    &= f\pi([w,v]_{|D}) + \pi((w(\tilde{f})\cdot v)_{|D})
\end{align*}
and the conclusion follows from the fact that $\xi(f) = w(\tilde{f})_{|D}$ for any choice of $w \in T_X(- \log D)$ and $\tilde{f} \in \corO_X$.

\end{proof}

\begin{remark} \label{rem:residue}
By this theorem, one obtains a new interpretation of the residue of a logarithmic form:  in fact, the residue map \eqref{ses:omega_log} coincides with the composition of the restriction $\Omega^1_X(\log D) \to \Omega^1_{L_{D/X}}$ with the map $R^*: \Omega^1_{L_{D/X}} \to \corO_D$ of \eqref{ses:atiyah_dual}.
\end{remark}

\begin{remark}
In \cite{tutti}, we defined a Hodge structure on the cohomology of $\atiyah(N)$ for $N$ a line bundle on a compact K\"ahler manifold $Y$. In the light of Theorem \ref{thm:simpson}, this Hodge structure is induced by the restriction of the mixed Hodge structure of $H^\cdot(X; \Omega^\bullet_N(\log Y))$.
\end{remark}

\begin{corollary} \label{cor:class}
The extension class of $L_{D/X}$ is equal to the Chern class of $N_{D/X}$. In particular, $L_{D/X}$ splits as a direct sum of Lie algebroids $T_D \oplus \corO_D$ if and only if the Chern class of $N_{D/X}$ is zero.
\end{corollary}

\vspace{0.4cm}

Now, let us consider a logarithmic connection $(E,\nabla)$. This is a $T_X(- \log D)$-connection, so the restriction of $\nabla$ to $D$ will define a $\atiyah(N_{D/X})$-connection on $E_{|D}$. This can be checked by hand, but also follows from general result on pull-back of structures of Lie algebroid module established in \cite{chemla}.

%Moreover, the restriction of a flat logarithmic connection will carry a structure of $\atiyah(N_{D/X})$-module.

By Remark \ref{rem:residue}, the residue of a logarithmic connection coincides with the isotropy endomorphism of the $\atiyah(N_{D/X})$-connection structure of the restriction. Let us summarize this with the following:

\begin{proposition} \label{prop:residue-isotropy}
Let $(E,\nabla)$ be a logarithmic connection, and $(V,\delta) = (E,\nabla)_{|D}$ its restriction to $D$, with the structure of $\atiyah(N_{D/X})$-connection.

Then the residue $\res_D \nabla$ coincides with $\delta_R$, the isotropy endomorphism of $\delta$. 
\end{proposition}

\begin{remark} \label{rem:sabbah}
This proposition generalizes other definitions of the residue of a logarithmic connection. See for instance Chapter 0 of \cite{sabbah}. In particular, in Chapter 0.14 of loc. cit., given a logarithmic connection $(E,\nabla)$, one constructs a $T_D$-connection on $E_{|D}$ under the assumption that the normal bundle $N_{D/X}$ is trivial. This construction follows from Proposition \ref{prop:residue-isotropy}, since after Corollary \ref{cor:class} when $N_{D/X}$ is trivial the Lie algebroid $\atiyah(N_{D/X})$ splits as $T_D \oplus \corO_D$, and in this case a $\atiyah(N_{D/X})$-connection on a $\corO_D$-module $V$ coincides with a pair $(R,\delta)$, with $R \in \End_{\corO_D}V$ and $\nabla$ a $T_D$-connection on $V$.

%One may see this as a special case of Proposition \ref{prop:residue-isotropy} as a generalizationo

%The $L_{D/X}$-connection structure on $E_{|D}$ explains this construction, since when $N_{D/X}$ is the trivial line bundle its Atiyah algebroid $L_{D/X}$ splits as the direct sum of the Lie algebroids $T_D \oplus \corO_D$, thus in this case the datum of a $L_{D/X}$-connection is equivalent to the datum of a $T_D$-connection and a $\corO_D$-connection (that is a $\corO_D$-linear endomorphism).
\end{remark}

\section{Monodromies and the Riemann-Hilbert correspondence}

\subsection{Monodromy of $\atiyah(N)$-modules}

Recall the following: let $Y$ be a complex manifold, and let $(E,\nabla)$ be a flat $T_Y$-connection of rank $r$. Then the sheaf of $\nabla$-horizontal sections $E^\nabla$ is a local system on $X$ with fiber $\mathbb C^r$.
\begin{definition}
The {\em monodromy} of $(E,\nabla)$ is the representation $T: \pi_1(Y) \to \GL(r,\mathbb C)$ defined by the local system $E^\nabla$.
\end{definition}

The monodromy of a flat vector bundle completely characterize it, since one has:
\begin{theorem} \label{thm:RHtrivial}
There is a one-to-one correspondence between:
\begin{itemize}
\item[-] equivalence classes of flat connections on $Y$ of rank $r$;
\item[-] conjugacy classes of representations of the fundamental group $T: \pi_1(Y) \to \GL(r,\mathbb C)$.
\end{itemize}
\end{theorem}

As it was noted in \cite{gualtieri}, this correspondence may be understood as a consequence of Theorem \ref{thm:LieII} and Theorem \ref{thm:morita}: flat connections are representations of the tangent Lie algebroid $T_Y$; the $\textsc{s}$-connected $\textsc{s}$-simply connected Lie algebroid integrating $T_Y$ is the fundamental groupoid $\Pi_1(Y)$, so by Theorem \ref{thm:LieII} flat connections are in correspondence with representations of $\Pi_1(Y)$; now, $\Pi_1(Y)$ is a transitive Lie groupoid, so by Theorem \ref{thm:morita} representations of $\Pi_1(Y)$ are in correspondence with representations of $\Pi_1(Y)^y_y = \pi_1(Y,y)$.

\vspace{0.4cm}
Let us examine the same arguments for the Atiyah algebroid of a line bundle: let  $N$ be a line bundle on $Y$, and consider $\atiyah(N)$, the Atiyah algebroid of $N$. The principal $\mathbb{C^*}$-bundle of frames of $N$ is isomorphic to $N\setminus Y$, with the $\mathbb{C}^*$-action given by dilation on the fibers, and $\atiyah(N)$ is isomorphic to $\atiyah(N\setminus Y)$. The $\textsc{s}$-connected $\textsc{s}$-simply connected Lie groupoid integrating $\atiyah(N\setminus Y)$ is $\mathfrak{C}(N \setminus Y)$ of Example \ref{ex:GP-groupoid}, which is transitive. So  representations of $\atiyah(N)$ are in correspondence with representations of the vertex group $\mathfrak{C}(N \setminus Y)^y_{y}$, for any $y\in Y$. Let us sum this up:
\begin{theorem} \label{thm:LieII_atiyah}
There is a one-to-one correspondence between:
\begin{itemize}
\item[-] equivalence classes of rank $r$ representations of the Lie algebroid $\atiyah(N)$;
\item[-] conjugacy classes of group homomorphisms $\mathfrak{C}(N\setminus Y)^y_y \to \GL(r, \mathbb{C})$. 
\end{itemize}
\end{theorem}
One can describe the vertex groups $\mathfrak{C}(N\setminus Y)^y_y$ as follows: recall that one has the diagram \eqref{diagram0} describing the gauge-path groupoid of a principal bundle. By pulling it back to a point $y \in Y$ one obtains the following diagram of groups:
\begin{equation}\label{diagram1}
\xymatrix{
& 0 \ar[d] & 0 \ar[d] & & \\
0 \ar[r] & \mathbb{Z} \ar[r]\ar[d] & \mathbb{C} \ar[r] \ar[d] & \mathbb{C}^* \ar[r] \ar@{=}[d] & 1 \\
1 \ar[r] & \pi_1(N\setminus Y, u) \ar[r]\ar[d] & \mathfrak{C}(N\setminus Y)^y_y \ar[r]\ar[d] & \mathbb C^* \ar[r] & 1 \\
& \pi_1(Y,y) \ar[d]\ar@{=}[r] & \pi_1(Y,y) \ar[d] & &\\
& 1 & 1 & & 
}\end{equation}
where $u$ is any point in the fiber $(N\setminus Y)(y)$. The first column is the exact sequence that describes the fundamental group of $N\setminus Y$, where the inclusion $\mathbb Z \to \pi_1(N\setminus Y, u)$ is given by the class of the loop around the zero on the fiber $(N\setminus Y)(y)$. Moreover, this is a push out diagram, and one obtains:
\begin{lemma} \label{lem:vertex_group}
Let us denote by $\gamma_0$ the image of $1$ via $\mathbb Z \to \pi_1(N\setminus Y, u)$.

Then there is a natural isomorphism of groups 
$$
\mathfrak{C}(N \setminus Y)^y_y \ \ = \frac{\pi_1(N\setminus Y , u) \times \mathbb C}{ \mathbb Z}
$$ 
where the action of $\mathbb Z$ is given by $[\gamma,z] = [\gamma \cdot \gamma_0^n , z-n]$ for any $n\in \mathbb Z$.
\end{lemma}

We can now state the following:
\begin{definition} \label{def:monodromy}
Let $(V,\delta)$ be a $\atiyah(N)$-module. 

The {\em monodromy} of $(V,\delta)$ is the representation $T_\delta:\pi_1(N\setminus Y) \to \GL( r, \mathbb{C})$ obtained from the composition $\pi_1(N\setminus Y) \to \mathfrak{C}(N\setminus Y)^y_y \to \GL(r,\mathbb{C})$.
\end{definition}

In the next subsection we will give a geometric construction of this monodromy, that justifies the terminology.

%\begin{lemma}
%The monodromy groups $\mathfrak{C}(N\setminus Y)_y$ fit in an exact sequence:
%$$
%1 \to \pi_1(N\setminus Y; u) \to \mathfrak{C}(N\setminus Y)_y \to \mathbb{C^*} \to 1 
%$$
%for any $y\in Y$ and $u \in N\setminus Y$ with $p(u)=y$.
%\end{lemma}
%\begin{proof}
%Let $[\gamma]$ be the class of an element in $\mathfrak{C}(N\setminus Y)_y$, and $u\in N\setminus Y$ with $p(u)=y$. Then $\gamma$ is a path in $N\setminus Y$ with $p(\gamma(0)) = p(\gamma(1)) = y$, and since $[\gamma] = [g \gamma]$, we can assume that $\gamma(0) = u$. Then the morphism $\mathfrak{C}(N\setminus Y)_y \to \mathbb{C}^*$ assigns to $[\gamma]$ the element in $g\in \mathbb{C}^*$ such that $\gamma(1) = g\gamma(0)$. Clearly the kernel of this map corresponds to loops that have $\gamma(0) = \gamma(1)$, that corresponds to loops in $N \setminus Y$ with base point $u$, i.e. elements of $\pi_1(N\setminus Y; u)$.

%\end{proof}

\subsection{From $\atiyah(N)$-modules to logarithmic connections}

Let $p: N \to Y$ be a line bundle over a complex variety as before. Consider $N$ as a manifold itself, and denote by  $\iota: Y \rightarrow N$ the embedding given by the zero section; this makes $Y$ a smooth divisor in $N$, and we can consider $T_N(-\log Y)$, the sheaf of vector fields on $N$ tangent to $Y$.

Consider $T_p$, the differential of the projection $p$; this gives the exact sequence
$$
0\rightarrow \ker (T_p) \rightarrow T_N \rightarrow p^* T_Y \rightarrow 0 .
$$
Since $N$ is a line bundle, the vertical bundle $\ker (T_p)$ is isomorphic to $\corO_N(Y)$, the line bundle whose sections are the meromorphic function on $N$ with a simple pole at $Y$.

Now consider the pull back of this sequence via the inclusion of  $T_N(-\log Y)$ in $T_N$. The restriction of $T_p$ to $T_N(-\log Y)$ is again surjective, and one has the exact sequence
\begin{equation} \label{ses:tangentlog}
0 \rightarrow \ker(T_p) \cap T_N(-\log Y) \rightarrow T_N(-\log Y) \rightarrow p^* T_Y \rightarrow 0 .
\end{equation}
Vectors in $\ker (T_p)\cap T_N(-\log Y)$ are tangent both to $Y$ and to the fibers of $p$, so it consists of those vector fields in $\ker(T_p)$ that vanish on $Y$. Then $\ker(T_p) \cap T_N(-\log Y) = \ker(T_p) \otimes \corO_N(-Y) = \corO_N$, i.e. $T_N(-\log Y)$ is an extension of $p^*T_Y$ by $\corO_N$. We have:

\begin{theorem} \label{thm:pull-back_atiyah}
The exact sequence \eqref{ses:tangentlog} is equivalent to the extension
$$
0 \rightarrow \corO_N \rightarrow p^*\atiyah(N) \rightarrow p^*T_Y \rightarrow 0
$$
obtained by pulling back to $N$ the short exact sequence of the Atiyah algebroid of $N$.
\end{theorem}
\begin{proof}
Since both $p^* \atiyah(N)$ and $T_N(-\log Y)$ are an extension of $p^* T_Y$ by $\corO_N$, to prove that they are equivalent extensions it suffice to construct a morphism of extensions $p^*\atiyah(N) \rightarrow T_N(-\log Y)$.

So let $D$ be a section of $\atiyah(N)$ and $a\in \corO_N$; they define $a\otimes D$, a section of $p^*\atiyah(N) = \corO_N \otimes_{p^{-1} \corO_Y} p^{-1} \atiyah(N)$. We are going to construct $\Psi(a\otimes D) \in T_N(-\log Y)$, where we see the latter as the sheaf of derivations of $\corO_N$ preserving the ideal $I_Y$.

Let us fix $s \in N$ a non vanishing local section. Then $Ds = A^s \cdot s$ for some $A^s \in \corO_Y$, and for any section $s'$ of $N$ we have $s' = fs$ for some $f\in\corO_Y$ and
\begin{equation}
D(fs) = (f A^s + \sigma_D(f)) s \ ,
\end{equation}
where recall from Definition \ref{def:atiyah_vb} that $\sigma_D$ denotes the symbol of $D$, and defines the anchor $\atiyah(N) \to T_Y$.

On the other hand, the algebra of functions $\corO_N$ is isomorphic to the completion of $\Sym_{\corO_Y}^\bullet N^*$, the symmetric algebra of $N^*$. By letting $s^*\in N^*$ denote the dual section to $s \in N$, in the neighborhood where $s$ is defined the elements of $\corO_N$ are formal power series in $s^*$ with coefficients in $\corO_Y$. Define the derivation $\delta = \Psi(a\otimes D)$ by
\begin{equation}
\delta(f) = a\cdot \sigma_D(f) \ \ \text{for } \ f\in\corO_Y,\quad \delta(s^*) = -a A^s s^*,
\end{equation}
and extend it to $\corO_N$ by the Leibniz rule. Under the isomorphism $\corO_N = \widehat{\Sym_{\corO_Y} N^*}$, the ideal $I_Y$ correspond to the ideal generated by $s^*$, so it is clear that $\Psi(a\otimes D) \in T_N(-\log Y)$. We should check that this definition does not depend on the choice of $s$, and that $\Psi$ defines a morphism of extensions.

Let $\bar{s}$ be another non vanishing section of $N$, with $\bar{s}= gs$ for some $g\in \corO_Y$ invertible function. Then $\bar{s}^* = g^{-1} s^*$ and $A^{\bar{s}} = A^s + g^{-1}\sigma_D(g)$. Then on one hand
\begin{align*}
\delta(\bar{s}^*) = - A^{\bar{s}} \bar{s}^* = -(A^s + g^{-1} \sigma_D(g)) g^{-1} s^*,
\end{align*}
while on the other hand
\begin{align*}
\delta(g^{-1} s^*) = \delta(g^{-1})s^* + g^{-1}\delta(s^*) = -g^{-2} \sigma_D(g) s^* - g^{-1} A^s s^*,
\end{align*}
so that $\delta$ is well defined.

The compatibility of $\Psi$ with the projection to $p^*T_Y$ is straightforward, since the projection of $a \otimes D$ is $a\otimes \sigma_D$, while the projection of $\delta$ is induced from the action of $\delta$ on the pull back of functions of $Y$, and for $\delta = \Psi(a \otimes D)$ this gives exactly $a \otimes \sigma_D$.

To check that $\Psi$ is compatible with the inclusion of $\corO_N$, remark that the inclusion $\corO_N \rightarrow T_N(-\log Y)$ corresponds to the Euler vector field on $N$, i.e. the vector field generated by the $\bC^*$-action on the fibers of $N$; let us denote it by $\delta_R$. Since the $\mathbb C^*$-action preserves the fibers, $\delta_R$ applied to  pull-backs of functions on $Y$ is zero, i.e. $\delta_R (f)=0$ for any $f\in \corO_Y$. By construction, $\delta_{\Psi(a \otimes D)}(f) = a \sigma_D (f)$, and this is zero exactly when $\sigma_D=0$, that is when $a\otimes D$ lives in the image of $\corO_N \to p^*L_N$.

\end{proof}

So the pull back $p^*\atiyah(N)$ is naturally a Lie algebroid over $N$, which is equivalent to the logarithmic tangent bundle $T_N(-\log Y)$. Thus, by the results of \cite{chemla} we have an induced functor at the level of modules, namely
$$
p^* : \ \ \atiyah(N)\text{-modules} \ \longrightarrow \ T_N(-\log Y)\text{-modules}\ .
$$
Since $p \circ \iota = \id_Y$, the $p^*$-functor is a left adjoint to the restriction functor
$$
\cdot_{|Y} : T_N(-\log Y)\text{-modules} \ \longrightarrow \ \atiyah(N)\text{-modules}\ .
$$

%Since $\Psi$ is a morphism of extensions, it is invertible, and its inverse $\Psi^{-1}:T_N(-\log Y) \rightarrow p^* \caL_N$ defines a morphism of Lie algebroids according to the definition of Section 2. Thus, by the results of \cite{chemla}, given a $\caL_N$-module $(V,\delta)$ one can pull-back $\delta$ to define a $T_N(-\log Y)$-module structure on the $\corO_N$-module $p^*V$.

% the pull-back   we obtain a pull back functor $p^*$ from $L_N$-modules to $T_N(-\log Y)$-modules.
%On the other hand, we have also the restriction functor $\iota^*$ from the category of $T_N(-\log Y)$-modules to the category of $L_N$-modules. Clearly, since $p\circ \iota = \id_Y$, we have that $\iota^* \circ p^*(V,\nabla)$ is naturally isomorphic to $(V,\nabla)$ itself. 

Let $(V,\delta)$ be a $\atiyah(N)$-module, and $(E,\nabla) = p^*(V,\nabla)$ its pull-back. This is a logarithmic connection with poles along $Y$, so it defines a smooth flat connection $(E_U,\nabla_U)$ on $U = N\setminus Y$. By Theorem \ref{thm:RHtrivial}, one has the monodromy of $(E_U,\nabla_U)$ that is a representation $T_{\nabla_U}: \pi_1(N\setminus Y) \to \GL(r,\mathbb C)$. 

On the other hand, we have another notion of monodromy of a $\atiyah(N)$-module, given in Definition \ref{def:monodromy}. These two notions coincide:

\begin{proposition}
Let $(V,\delta)$ be a $\atiyah(N)$-module. Let us denote by $T_{\nabla_U}$ the monodromy of the flat connection obtained by restricting to $U = N\setminus Y$ the pull back $p^*(V,\delta)$, and by $T_\delta$ the monodromy of the $\atiyah(N)$-module as defined in Definition \ref{def:monodromy}.

Then $T_{\nabla_U}$ coincides with $T_\delta$.
\end{proposition}

\begin{proof}
Let us examine more closely how $T_\delta$ is defined. For $(V,\delta)$ a $\atiyah(N)$-module, for any $g \in \mathfrak{C}(N\setminus Y)$ one can define the parallel transport along $g$ defined by $\delta$, which is an isomorphism $\tau_g: V(s(g)) \to V(t(g))$ (cf. \cite{cf-integrability}). This defines the $\mathfrak{C}(N\setminus Y)$-module structure on $V$ that is granted from the second theorem of Lie.

Now, by definition $\mathfrak{C}(N\setminus Y) = \Pi_1(N\setminus Y) / \mathbb{C}^*$, and this underlies the morphism of Lie groupoids $([\cdot],p_\circ):(\Pi_1(N\setminus Y) , N\setminus Y \to (\mathfrak{C}(N\setminus Y) , Y)$. This induces a functor
$$
p_\circ^* : \ \ \mathfrak{C}(N\setminus Y)\text{-modules} \longrightarrow  \Pi_1(N\setminus Y)\text{-modules}\ ,
$$
and the monodromy $T_\delta$ coincides with the monodromy of $p_\circ^*(V,\tau)$, which is a flat connection on $N\setminus Y$. Now, $p_\circ^*(V,\tau) = p^*(V,\delta)_{|U}$, so the monodromies $T_\delta$ and $T_{\nabla_U}$ coincide.
\end{proof}

\subsection{A Riemann-Hilbert correspondence for $\atiyah(N)$-modules}

We now give an important application of the correspondence established in the previous section: by using the correspondence of Theorem \ref{thm:LieII_atiyah} and some simple argument, we will prove a Riemann-Hilbert correspondence for $\atiyah(N)$-modules, and show that the Riemann-Hilbert correspondence for regular meromorphic connections may be obtained as a pull-back of the equivalent results for $\atiyah(N)$-modules.

\vspace{0.4cm}
Let us recall the results on the Riemann-Hilbert correspondence for regular meromorphic connections. The main point is the existence of logarithmic connections extending a given flat connection defined on the complement of a divisor. This may be stated as follows (cf. \cite{deligne}):
\begin{theorem} \label{deligne}
Let $X$ be a complex manifold, and $D$ a smooth divisor in $X$. Let $T$ be a representation $\pi_1(X\setminus D) \to \GL(r,\mathbb C)$, and $\tau: \mathbb{C}^* \to \mathbb C$ a splitting of the exact sequence 
$$
\xymatrix{
0 \ar[r] & \mathbb Z \ar[r] & \mathbb C \ar[r]^{e^{2\pi i \cdot}} & \mathbb C^* \ar[r] & 1\ . 
%0 \longrightarrow \mathbb Z \longrightarrow \mathbb C \stackrel{\exp(2\pi i \cdot)}{\longrightarrow} \mathbb C^* \longrightarrow 0\ . 
}
$$

Then there exist a unique flat logarithmic connection $(E^\tau, \nabla^\tau)$ on $X$ with poles along $D$ such that 
\begin{itemize}
\item[-] the monodromy of $(E^\tau,\nabla^\tau)_{|U}$ is $T$;
\item[-] the eigenvalues of the residue of $\nabla^\tau$ along $D$ live in the image of $\tau$.
\end{itemize}
\end{theorem}

One calls the logarithmic connections $(E^\tau, \nabla^\tau)$ {\em Deligne's extensions} of $T$, and the {\em canonical extension} of $T$ is the extension $(E^\tau,\nabla^\tau)$ with $\tau$ such that $\operatorname{Im}(\tau) = [0,1) \times \mathbb{R}i$. 

\begin{definition}
Let $\corO_X(*D) = \bigcup_{n\geq 0} \corO_X(nD)$ be the sheaf of meromorphic functions on $X$ with poles along $D$. A {\em meromorphic bundle} with poles along $D$ is a locally free $\corO_X(*D)$-module.

A {\em $\corO_X$-lattice} of a meromorphic bundle $\belE$ is a subsheaf $E\subseteq \belE$ which is a locally free $\corO_X$-module and such that $\belE = E \otimes_{\corO_X} \corO_X(*D)$.

A flat connection on a meromorphic bundle $\nabla: \belE \to \belE \otimes \Omega^1_X$ is said {\em regular} if there exist a $\corO_X$-lattice $E$ of $\belE$ such that $\nabla(E) \subseteq E \otimes \Omega^1_X(\log D)$.
\end{definition}

Given a logarithmic connection $(E,\nabla)$, one has a naturally associated regular flat connection on a meromorphic bundle, namely $\belE = E \otimes \corO_X(*D)$ with $\nabla(s \otimes f) = \nabla s \otimes f + s \otimes df$, where one is using the natural inclusion $E \otimes \Omega^1_X(\log D) \subseteq E \otimes \corO_X(D) \otimes \Omega^1_X$.

Given a representation $T: \pi_1(X\setminus D) \to \GL(r, \mathbb C)$ and two splittings $\tau$ and $\sigma$ the associated Deligne's extension $(E^\tau,\nabla^\tau)$ and $(E^\sigma,\nabla^\sigma)$ have the same associated regular flat meromorphic connection, and one obtains the following Riemann-Hilbert correspondence for regular singular connection (cf.  \cite{steenbrink}, Theorem 11.7):
\begin{theorem}[Riemann-Hilbert correspondence] \label{thm:RHDeligne}
Let $X$ be a complex manifold and $D$ a smooth divisor in $X$. Then there is a one-to-one correspondence between
\begin{itemize}
\item[-] conjugacy classes of representations of the fundamental group $\pi_1(X\setminus D)$;
\item[-] equivalence classes of regular flat meromorphic connections on $X$ with singularities along $D$.
\end{itemize}
\end{theorem}

\begin{remark}
These theorems hold also for $D$ a divisor with simple normal crossings, we state them in the case $D$ smooth since this is what we need in the following. Extension of the results of this paper to the simple normal crossing case is a work in progress.
\end{remark}

\vspace{0.4cm}
%Now we will show that these results, locally around the divisor $D$, follow directly from Theorem \ref{thm:pull-back_atiyah} and the second theorem of Lie: 
%in fact we show that if $X=N$ is a line bundle on a manifold $Y$, and we see $Y$ as a smooth divisor in $N$ via the zero section embedding, then the canonical extension 
We now prove the following analogue of Theorem \ref{deligne} for $\atiyah(N)$-modules:
\begin{theorem} \label{tau-atiyah}
Let $Y$ be a complex manifold, and $N$ a line bundle on it. Let $T$ be a representation of $\pi_1(N\setminus Y) \to \GL(r,\mathbb C)$, and $\tau: \mathbb{C}^* \to \mathbb C$ a splitting of $\exp: \mathbb C \to \mathbb C^*$.

Then there exist a unique $\atiyah(N)$-module $(V^\tau, \delta^\tau)$ such that 
\begin{itemize}
\item[-] the monodromy of $(V^\tau,\delta^\tau)$ is $T$;
\item[-] the eigenvalues of the isotropy endomorphism $\delta^\tau_R$ live in the image of $\tau$.
\end{itemize}
\end{theorem}
\begin{proof}
Fix $y \in Y$ and $u \in (N\setminus Y)(y)$. We will construct $(V^\tau,\delta^\tau)$ from the associated representation of the vertex group $\mathfrak{C}(N\setminus Y)_y^y$ under the correspondence of Theorem \ref{thm:LieII_atiyah}. 
%Recall that $\mathfrak{C}(N\setminus Y)_y^y$ is described by the diagram \eqref{diagram1}. 

The splitting $\tau$ induces a splitting (that denote by ${\tau}$ as well) of the sequence 
$$
0 \longrightarrow \mathbb{Z}\cdot \id_V \longrightarrow \mathfrak{gl}(r,\mathbb C) \stackrel{e^{2\pi i \cdot}}{\longrightarrow} \GL(r,\mathbb C) \longrightarrow 1
$$
determined by ${\tau}(A)$ equal to the only $B\in \mathfrak{gl}(r,\mathbb C)$ such that $\exp(2\pi i B) = A$ and such that the eigenvalues of $B$ lie in the image of $\tau$ in $\mathbb C$. So we need to show that given $T$ and $\tau$ as in the statement of the proposition, there exist a unique group morphism $\tilde{T}: \mathfrak{C}(N\setminus Y)_y^y \to \GL(r, \mathbb C)$ such that:
\begin{enumerate}
\item the diagram
$$
\xymatrix{
\pi_1(N\setminus Y,u) \ar[r]^T \ar[d] & \GL(r,\mathbb C) \ar@{=}[d] \\
\mathfrak{C}(N\setminus Y)_y^y \ar[r]^{\tilde{T}} & \GL(r,\mathbb C)
}
$$
commutes;
\item the composition 
$$
\mathbb Z \hookrightarrow \pi_1(N\setminus Y, u) \stackrel{T}{\to} \GL(r,\mathbb C) \stackrel{\tau}{\to} \mathfrak{gl}(r,\mathbb C)
$$
coincides with the map $\mathbb Z \hookrightarrow \mathbb C \to \mathfrak{gl}(r,\mathbb C)$ induced by $\mathbb C \to \mathfrak{C}(N\setminus Y)_y^y \stackrel{\tilde{T}}{\to} \GL(r,\mathbb C)$. 

\end{enumerate}

Recall that $\mathfrak{C}(N\setminus Y)_y^y$ is a push out, and may be described as in Lemma \ref{lem:vertex_group}. Denote by $\gamma_0$ the image of $1$ via $\mathbb Z \hookrightarrow \pi_1(N\setminus Y,u)$, let $A_0 = T(\gamma_0)$ and $B_0 = \tau(A_0)$. Define 
$$
\tilde{T}([\gamma,z]) = T(\gamma) \cdot \exp(2\pi i z B_0)\ .
$$
It is easy to check that, since $\gamma_0$ is central in $\pi_1(N\setminus Y,u)$, this is a well defined morphism of groups, and satisfies the conditions 1. and 2. above.

The uniqueness of the $\atiyah(N)$-module $(V^\tau,\nabla^\tau)$ follows from Proposition \ref{prop:equivalent_monodromy} that we will show soon.
\end{proof}

Now, taking the pull-back via $p^*$ of the $\atiyah(N)$-modules $(V^\tau,\delta^\tau)$ one obtains a logarithmic connection that has monodromy $T_\delta$ and whose residue has eigenvalues that lie in the image of $\tau$. Then $p^*(V^\tau,\delta^\tau)$ must coincide with the Deligne's extension $(E^\tau,\nabla^\tau)$ of $T_\delta$. So we have shown:
\begin{corollary}
Let $N$ be a line bundle on $Y$, and $T$ and $\tau$ as before in the statement of Theorem \ref{tau-atiyah}. 

Then the Deilgne extension $(E^\tau, \nabla^\tau)$ of Theorem \ref{deligne} coincide with $p^*(V^\tau,\delta^\tau)$, where $(V^\tau,\delta^\tau)$ are the $\atiyah(N)$-modules of Theorem \ref{tau-atiyah}.
\end{corollary}

Now, let us characterize $\atiyah(N)$-modules having the same monodromy. Via the correspondence of Theorem \ref{thm:LieII_atiyah}, the $\atiyah(N)$-modules with trivial monodromy correspond to the representations of $\mathfrak{C}(N\setminus Y)^y_y$ that are induced by a representation $\mathbb{C}^* \to \GL(r,\mathbb{C})$. Representations of the form $\lambda \mapsto \lambda^a \cdot \id$ for $a\in \mathbb Z$ correspond to the $\atiyah(N)$-module $N^{\otimes a}$, and one obtains:

\begin{proposition} \label{prop:equivalent_monodromy}\ 

\begin{itemize}
\item[-]
The irreducible $\atiyah(N)$-modules with trivial monodromy are $N^{\otimes a}$, for $a\in\mathbb Z$.
\item[-]
Two irreducible $\atiyah(N)$-modules $V$ and $V'$ have the same monodromy if and only if $V' = V \otimes N^{\otimes a}$ for some $a\in\mathbb Z$.
\end{itemize}
\end{proposition}

We can reformulate this proposition in the following way:
\begin{theorem}[Riemann-Hilbert correspondence for $\atiyah(N)$-modules]
Let $N$ be a line bundle over a complex manifold $Y$. Then there is a one-to-one correspondence between:
\begin{itemize}
\item[-] conjugacy classes of representations of the fundamental group $\pi_1(N\setminus Y)$;
\item[-] monodromy classes of $\atiyah(N)$-modules, where two irreducible $\atiyah(N)$-modules $V$ and $V'$ belong to the same monodromy class if and only if $V' = V \otimes N^{\otimes a}$ for some $a\in \mathbb Z$.
\end{itemize}

\end{theorem}

We interpret this result as a Riemann-Hilbert correspondence, since it is a characterization of the $\atiyah(N)$-modules having a prescribed monodromy. Moreover, this implies Theorem \ref{thm:RHDeligne} in the case $X=N$ and $D= Y$, since for any $a\in \mathbb Z$ the pull-back of $N^{\otimes a}$ yields the trivial regular flat meromorphic connection $(\corO_X(*D), d)$.

\end{document}